\documentclass[reqno,english]{amsart}
\usepackage{etex}
\usepackage{amsmath,amssymb,amsthm,bbm,mathtools,comment}
\usepackage[shortlabels]{enumitem}
\usepackage[pdftex,colorlinks,backref=page,citecolor=blue]{hyperref}
\hypersetup{pdfpagemode=UseNone,pdfstartview={XYZ null null 1.00}}
\usepackage[mathscr]{euscript}
\usepackage{tikz,babel,adjustbox}
\usetikzlibrary{decorations.pathreplacing,calc,arrows,arrows.meta,patterns}
\usepackage[final]{microtype}
\usepackage[numbers]{natbib}
\usepackage{cmtiup}
\usepackage{amsfonts}
\usepackage{graphicx}
\usepackage{caption}
\usepackage{verbatim}
\usepackage{array}
\usepackage[frame,cmtip,arrow,matrix,line,graph,curve]{xy}
\usepackage{graphpap, color, pstricks}
\usepackage[final]{microtype}
\usepackage{cmtiup}

\allowdisplaybreaks

\theoremstyle{plain}
\newtheorem{theorem}{Theorem}[section]
\newtheorem{lemma}[theorem]{Lemma}
\newtheorem{claim}[theorem]{Claim}
\newtheorem{proposition}[theorem]{Proposition}
\newtheorem{corollary}[theorem]{Corollary}
\newtheorem{conjecture}[theorem]{Conjecture}
\newtheorem{obs}[theorem]{Observation}

\newtheorem{Remark}[theorem]{Remark}
\theoremstyle{remark}

\newcommand{\Z}{\mathbb{Z}}

\newcommand{\QQ}{\mathbb{Q}}
\newcommand{\pr}{\mathbb{P}}

\newcommand{\E}[0]{\mathbb{E}}
\newcommand{\Ee}[0]{{\sf E}}

\newcommand{\sub}[0]{\subseteq}
\newcommand{\sm}[0]{\setminus}
\renewcommand{\dots}[0]{,\ldots,}

\newcommand{\ov}[0]{\overline}

\newcommand{\beq}[1]{\begin{equation}\label{#1}}
\newcommand{\enq}[0]{\end{equation}}

\newcommand{\f}[0]{{\mathcal F}}

\newcommand{\g}[0]{{\mathcal G}}
\newcommand{\h}[0]{{\mathcal H}}

\newcommand{\K}[0]{{\mathcal K}}

\newcommand{\ra}[0]{\rightarrow}
\newcommand{\Ra}[0]{\Rightarrow}

\newcommand{\Lra}[0]{\Leftrightarrow}

\newcommand{\0}[0]{\emptyset}

\newcommand{\C}[2]{\binom{#1}{#2}}

\newcommand{\ga}[0]{\alpha }
\newcommand{\gb}[0]{\beta }
\newcommand{\gc}[0]{\gamma }
\newcommand{\gd}[0]{\delta }
\newcommand{\gD}[0]{\Delta }

\newcommand{\gL}[0]{\Lambda}

\newcommand{\gO}[0]{\Omega}

\newcommand{\gs}[0]{\sigma}

\newcommand{\gz}[0]{\zeta}
\newcommand{\eps}[0]{\varepsilon }
\newcommand{\vt}[0]{\vartheta}
\newcommand{\vs}[0]{\varsigma}

\newcommand{\vp}[0]{\varphi}

\DeclareMathOperator{\Exp}{Exp}

\renewcommand{\1}[0]{\textbf{1}}

\newcommand{\mn}[0]{\medskip\noindent}
\newcommand{\nin}[0]{\noindent}

\title{Thresholds versus fractional expectation-thresholds}
\author{Keith Frankston}
\address{Department of Mathematics,	Rutgers University, Piscataway, NJ 08854, USA}
\email{keith.frankston@math.rutgers.edu}

\author{Jeff Kahn}
\address{Department of Mathematics, Rutgers University, Piscataway, NJ 08854, USA}
\email{jkahn@math.rutgers.edu}

\author{Bhargav Narayanan}
\address{Department of Mathematics, Rutgers University, Piscataway, NJ 08854, USA}
\email{narayanan@math.rutgers.edu}

\author{Jinyoung Park}
\address{Department of Mathematics, Rutgers University, Piscataway, NJ 08854, USA}
\email{jp1324@math.rutgers.edu}


\begin{document}

\maketitle
\begin{abstract}
Proving a conjecture of Talagrand, 
a fractional version of the ``expectation-threshold'' conjecture of Kalai and the second author, we show that for any increasing family $\f$ on a finite set $X$, we have $p_c (\f) =O( q_f  (\f) \log \ell(\f))$, where $p_c(\f)$ and $q_f(\f)$ are the threshold and ``fractional expectation-threshold'' of $\f$, and $\ell(\f)$ is the maximum size of a minimal member of $\f$.
This easily implies several heretofore difficult results
and conjectures in probabilistic combinatorics, 
including thresholds for perfect hypergraph matchings (Johansson--Kahn--Vu), bounded degree spanning trees (Montgomery), and bounded degree graphs (new). 
We also resolve (and vastly extend) the ``axial'' version of 
the random multi-dimensional assignment problem
(earlier considered by Martin--M\'{e}zard--Rivoire and Frieze--Sorkin).
Our approach builds on a recent breakthrough of Alweiss, Lovett, Wu and Zhang on the Erd\H{o}s--Rado
``Sunflower Conjecture.'' 
\end{abstract}
\maketitle

\section{Introduction}
Our most important contribution here is the proof of a conjecture of Talagrand~\citep{Talagrand} that is  
a fractional version of the ``expectation-threshold'' conjecture of 
Kalai and the second author~\citep{KK}.  
For an increasing family $\f$ on a finite set $X$, we write (with definitions below)
$p_c(\f)$, $q_f(\f)$ and $\ell(\f)$ for the threshold, fractional expectation-threshold, and size of a largest 
minimal element of $\f$.  In this language, our main result is the following.
\begin{theorem}\label{maintheorem}
There is a universal $K$ such that for every finite $X$ and increasing
$\f\sub 2^X$,
\[p_c (\f) \le K q_f  (\f) \log \ell(\f).\]
\end{theorem}
\nin
As observed below, $q_f(\f)$ is a more or less trivial lower bound on $p_c(\f)$, and Theorem~\ref{maintheorem} says this bound is never far from the truth. (Apart from the constant $K$, the upper bound is tight in many of the most interesting cases.) 

Thresholds have been a---maybe \emph{the}---central concern of the study of random discrete structures 
(random graphs and hypergraphs, for example) since its initiation by Erd\H{o}s and R\'enyi~\citep{ER}, with much 
work around identifying thresholds for specific properties (see~\citep{BBbook, JLR}), though it was 
not observed until~\citep{BT} that \emph{every} increasing $\f$ admits a threshold (in the 
Erd\H{o}s--R\'enyi sense; see below). See also~\citep{Friedgut2} for developments, since~\citep{Friedgut}, on the very interesting question of \emph{sharpness} of thresholds.

Our second main result is Theorem~\ref{TFS} below, which was motivated by 
work of Frieze and Sorkin~\citep{FS} on
the  ``random multi-dimensional assignment problem.''  The statement 
is postponed until we have filled in some background, to which we now turn.
(See the beginning of Section~\ref{prelim} for notation not defined here.)

\mn
\textbf{Thresholds.}
For a given $X$ and $p\in [0,1]$, $\mu_p$ is the product
measure on $2^X$ given by $\mu_p(S) = p^{|S|}(1-p)^{|X\sm S|}$.
An $\f\sub 2^X$ is {\em increasing} if
$B\supseteq A\in\f\Ra B\in\f$.
If this is true
(and
$\f\neq 2^X, \0$), then $\mu_p(\f)(:= \sum\{\mu_p(S): S \in \f\}$)
is strictly increasing in $p$,
and the {\em threshold}, $p_c(\f)$,
is the unique $p$ for which
$\mu_p(\f)=1/2$.
This is finer than the original Erd\H{o}s--R\'enyi notion,
according to which
$p^*=p^*(n)$ is \emph{\textbf{a}} threshold for $\f=\f_n$
if
$\mu_p(\f)\ra 0$ when $p\ll p^*$ and $\mu_p(\f)\ra 1$ when $p\gg p^*$.
(That $p_c(\f)$ \emph{is} always an Erd\H{o}s--R\'enyi threshold 
follows from~\citep{BT}.)

Following
\citep{Talagrand4, Talagrand1,Talagrand}, we say $\f$ is
$p$-{\em small} if there is a $\g\sub 2^X$ such that
$\f\sub \langle \g\rangle:=\{T: \exists S\in \g, S\sub T\}$
and
\beq{psmall2}
\mbox{$\sum_{S\in \g}p^{|S|} \leq 1/2$}.
\enq

\nin
Then
$q(\f):= \max\{\mbox{$p$ : $\f$ is $p$-small}\}
$, which we call the {\em expectation-threshold} of $\f$ (note the term is used slightly 
differently in~\citep{KK}), is a trivial lower bound
on $p_c(\f)$, since for $\g$ as above and $T$
drawn from $\mu_p$,
\beq{triv.l.b.}
\mbox{$\mu_p(\f) \leq \mu_p  (\langle \g\rangle)\leq
\sum_{S\in \g}\mu_p(T\supseteq S) =
\sum_{S\in \g}p^{|S|} ~~~(=\E[|\{S\in \g:S\sub T\}|]).$}
\enq
The following statement, the main conjecture (Conjecture~1) of 
\citep{KK}, says that
for {\em any} $\f$, this trivial lower bound on $p_c(\f)$ is close to the truth.
\begin{conjecture}\label{CKK}
There is a universal $K$ such that for every finite $X$ and increasing
$\f\sub 2^X$,
\[   
p_c(\f) \le Kq(\f)\log |X|.
\]   
\end{conjecture}

We should emphasize how strong this is
(from~\citep{KK}:  ``It would probably be more sensible to conjecture that it is \emph{not} true'').
For example, it easily implies---and was largely motivated by---Erd\H{o}s--R\'enyi thresholds for 
(a) perfect matchings in random $r$-uniform hypergraphs, 
and (b) appearance of a given bounded degree spanning tree in a random graph.
These have since been resolved:  the first---\emph{Shamir's Problem}, circa 1980---in~\citep{JKV},
and the second---a mid-90's suggestion of the second author---in~\citep{Montgomery}.
Both arguments are difficult and specific to the problems they address (e.g.\ they are 
utterly unrelated either to each other or to what we do here). 
See Section~\ref{Applications} for more on these and other consequences.

Talagrand~\citep{Talagrand4, Talagrand} suggests relaxing ``$p$-small'' by replacing 
the set system 
$\g$ above by what we may think of as a \emph{fractional} set system, $g$: say
$\f$ is {\em weakly $p$-small} if there is a $g:2^X\ra \mathbb{R}^+$
such that
\[
\mbox{$\sum_{S\sub T}g(S)\geq 1 ~~\forall T\in \f~~$ and
$~~\sum_{S \sub X}g(S)p^{|S|} \leq 1/2.$}
\]
Then
$
q_f(\f):= \max\{\mbox{$p$ : $\f$ is weakly $p$-small}\} $, the \emph{fractional expectation-threshold} of $\f$,
satisfies
\beq{triv.l.b.*}
q(\f)\leq q_f(\f)\leq p_c(\f)
\enq
(the first inequality is trivial and the second is similar to~\eqref{triv.l.b.}), 
and Talagrand~\citep[Conjectures~8.3 and 8.5]{Talagrand} proposes a sort of LP relaxation of Conjecture~\ref{CKK}, 
and then a strengthening thereof.
The first of these, the following, 
replaces $q$ by $q_f$ in Conjecture~\ref{CKK}; the second, which
adds replacement of $|X|$ by the smaller $\ell(\f)$, is our Theorem~\ref{maintheorem}.
\begin{conjecture}\label{CT}
There is a universal $K$ such that for every finite $X$ and increasing
$\f\sub 2^X$,
\[   
p_c(\f) \le Kq_f(\f)\log |X|.
\]   
\end{conjecture}


Talagrand further suggests the following ``very nice problem of combinatorics,'' which implies  
\emph{equivalence of} Conjectures~\ref{CKK} and~\ref{CT}, as well as of 
Theorem~\ref{maintheorem} and the corresponding strengthening of Conjecture~\ref{CKK}.
\begin{conjecture}\label{littleTal}
There is a universal $K $ such that, for any increasing $\f$ on a finite set $X$,
$
q(\f)\geq q_f(\f)/K.
$
\end{conjecture}
\nin
(That is,
weakly $p$-small implies $(p/K)$-small.)

Note the interest here is in Conjecture~\ref{littleTal} for its own sake and as the most likely
route to Conjecture~\ref{CKK}; all applications 
of the latter that 
we're aware of follow 
just as easily from Theorem~\ref{maintheorem}.

\mn
\textbf{Spread hypergraphs and spread measures.}
In this paper a \emph{hypergraph} on the (\emph{vertex}) set $X$ is a collection $\h$ of subsets of $X$
(\emph{edges} of $\h$),
with \emph{repeats allowed}.   
For $S \sub X$, we use 
$\langle S\rangle $ for $ \{T\sub X:T\supseteq S\}$, and for a hypergraph $\h$ on $X$, we write
$\langle\h\rangle$ for $\cup_{S\in \h}\langle S\rangle$.
We say $\h$ is $\ell$-\emph{bounded} (resp.\ $\ell$-\emph{uniform} or an \emph{$\ell$-graph}) 
if each of its members 
has size at most (resp.\ exactly) $\ell$,
and $\kappa$-\emph{spread} if 
\beq{spread}
|\h\cap \langle S\rangle|\leq \kappa^{-|S|}|\h| ~~\forall  S\sub X.
\enq
(Note that edges are counted with multiplicities on both sides of~\eqref{spread}.)

A major advantage of the fractional versions (Conjecture~\ref{CT} and Theorem~\ref{maintheorem}) 
over Conjecture~\ref{CKK}---and the source of
the present relevance of~\citep{ALWZ}---is that they admit, 
via linear programming duality, reformulations in which the specification of $q_f(\f)$ 
gives a usable starting point. Following~\citep{Talagrand}, we say 
a probability measure $\nu$ on $2^X$ is $q$-\emph{spread}
if 
\[
\nu(\langle S\rangle)\leq q^{|S|} ~~\forall  S\sub X.
\]
Thus a hypergraph $\h$ is $\kappa$-spread iff uniform measure on $\h$ is 
$q$-spread with $q=\kappa^{-1}$.

As observed by Talagrand~\citep{Talagrand}, the following is an easy consequence of duality.
\begin{proposition}\label{Tal6.7}
For an increasing family $\f$ on $X$, if $q_f(\f)\leq q$, then there is a $(2q)$-spread probability measure on $2^X$ supported on $\f$. \qed
\end{proposition}
\nin
This allows us to reduce Theorem~\ref{maintheorem} to the following alternate (actually, equivalent)
statement.  In this paper \emph{with high probability} (w.h.p.) means with probability tending to 1 
as $\ell\ra\infty$.

\begin{theorem}\label{MT}
There is a universal $K$ such that for any $\ell$-bounded, $\kappa$-spread hypergraph $\h$ on $X$, a uniformly random $((K\kappa^{-1}\log \ell)|X|)$-element subset of $X$ belongs to $\langle\h\rangle$ w.h.p.
\end{theorem} 
\nin
The easy reduction is given in Section~\ref{prelim}.

\nin
\textbf{Assignments.}
Our second main result provides upper bounds on
the minima of a large class of hypergraph-based stochastic processes,
somewhat in the spirit of~\citep{Talagrand3} (see also~\citep{Talagrand1, Talagrand2}),
saying that in ``smoother'' settings, the logarithmic corrections of
Conjecture~\ref{CT} and Theorem~\ref{maintheorem} are not needed. 

For a hypergraph $\h$ on $X$, 
let $\xi_x$ ($x\in X$) be independent random variables, each uniform from $[0,1]$,
and set 
\beq{xiH}
\xi_\h= \min_{S\in\h}\sum_{x\in S}\xi_x
\enq
and $Z_\h=\E [\xi_\h]$.

\begin{theorem}\label{TFS}
There is a universal $K $ such that for any $\ell$-bounded, $\kappa$-spread hypergraph $\h$, we have
$Z_\h \le K\ell/\kappa$, and
$\xi_\h \leq K\ell/\kappa$ w.h.p.
\end{theorem}
\nin
These bounds are often tight (again up to the value of $K$). The distribution of the $\xi_x$'s is not very important;
e.g.\ it's easy to see that 
the same statement holds if they are $\Exp(1)$ random variables, as in the next example.

Theorem~\ref{TFS} was motivated by work of 
Frieze and Sorkin~\citep{FS} on the ``axial'' version of the 
\emph{random d-dimensional assignment problem}. 
This asks (for fixed $d$ and large $n$) for estimation of 
\beq{minxi}
Z_{d}^A (n)= \E\left[\min\sum_{x\in S} \xi_x\right],
\enq
where the $\xi_x$'s ($x\in X:= [n]^d$) are independent $\textrm{Exp}(1)$ weights
and $S$ ranges over ``axial assignments,'' meaning $S\sub X$
meets each \emph{axis-parallel hyperplane} ($\{x\in X: x_i=a\}$ for some 
$i\in [d]$ and $a\in [n]$) exactly once.
For $d=2$ this is classical; see~\citep{FS} for its rather glorious history.
For $d=3$ the deterministic version was one of Karp's~\citep{Karp} original NP-complete 
problems.  Progress on the random version has been limited; see~\citep{FS} for a guide
to the literature.

\nin

Frieze and Sorkin show (regarding bounds; they are also interested in algorithms)
that for suitable $c_1>0$ and $ c_2$, 
\beq{FSP}
c_1 n^{-(d-2)}< Z^A_{d}(n) < c_2n^{-(d-2)}\log n.
\enq
(The lower bound is easy and the upper bound follows from the
Shamir bound of~\citep{JKV}.)

In present language, $Z_{d}^A (n)$ is essentially 
(that is, apart from the difference in the 
distributions of the $\xi_x$'s) $Z_\h$, with $\h$ the set of perfect matchings of
the complete, balanced $d$-uniform $d$-partite hypergraph on $dn$ vertices
(that is, the collection of $d$-sets meeting each of the pairwise disjoint $n$-sets
$V_1\dots V_d$).
This is easily seen to be $\kappa$-spread with $\kappa= (n/e)^{d-1}$ (apart from the nearly irrelevant $d$-particity, 
it is the $\h$ of Shamir's Problem), so the correct bound is an instance of 
Theorem~\ref{TFS}:
\begin{corollary}
\label{FSbd}
$Z_{d}^A (n) =\Theta (n^{-(d-2)})$.
\end{corollary}

Frieze and Sorkin also considered the ``planar'' version of the problem, in which $S$ in~\eqref{minxi}
meets each \emph{line} ($\{x\in X: x_j=y_j~\forall j\neq i\}$ for some 
$i\in [d]$ and $y\in X$) exactly once; and one may of course 
generalise from hyperplanes/lines to $k$-dimensional ``subspaces''
for a given $k\in [d-1]$. It's easy to see what to expect here, and
one may hope Theorem~\ref{TFS} will eventually apply, but we at present lack the technology to say the relevant $\h$'s are suitably spread (see Section~\ref{Concluding}).

\mn
\textbf{Organisation.} 
Section~\ref{prelim} includes minor preliminaries 
and the
derivation of Theorem~\ref{maintheorem} from Theorem~\ref{MT}.
The heart of our argument, Lemma~\ref{ML}, is proved in Section~\ref{SML};
our approach here strengthens that of
the recent breakthrough of Alweiss, Lovett, Wu and Zhang~\citep{ALWZ} on the Erd\H{o}s--Rado
``Sunflower Conjecture''~\citep{ErdosRado}.
Section~\ref{SJ} adds one small technical point (more or less repeated from~\citep{ALWZ}), and
the proofs of Theorems~\ref{MT} and~\ref{TFS} are given in Sections~\ref{SMT} and~\ref{SFS}. 
Finally, Section~\ref{Applications} outlines a few applications and
Section~\ref{Concluding} discusses unresolved questions.

\section{Little things}\label{prelim}

\nin
\textbf{Usage.}
As is usual, we use $[n] $ for $ \{1, 2, \dots, n\}$, $2^X$ for the power set of $X$, $\binom{X}{r}$ for the family of $r$-element subsets of $X$, and 
$[S,T]$ for $\{R:S\sub R\sub T\}$. Our default universe is $X$, with $|X|=n$. 

In what follows we assume $\ell$ and $n$ are somewhat large 
(when there is an $\ell$ it will be at most $n$),
as we may do since smaller values can by handled by adjusting the $K$'s in Theorems~\ref{MT} and~\ref{TFS}.
Asymptotic notation referring to some parameter $\lambda$ (usually $\ell$) is used in the natural way:  
implied constants in $O(\cdot)$ and $\gO(\cdot)$ are independent of $\lambda$, and 
$f=o(g)$ (also written $f\ll g$) means $f/g$ is smaller than any given $\eps>0$ for large enough values of 
$\lambda$.
Following a standard abuse, we usually pretend large numbers are integers.

For $p \in [0,1]$ and $m\in [n]$, $X_p$ and $X_m$ are (respectively) a $p$-random subset of $X$ 
(drawn from $\mu_p$) and a uniformly random $m$-element subset of $X$.
The latter is not entirely kosher, since we will also see \emph{sequences} $X_i$; but we will never see 
both interpretations in close proximity, and the overlap should cause no confusion.

\mn

In a couple places it will be helpful to assume uniformity, which we will justify using the next little point. 

\begin{obs}\label{Ounif}
If $\h$ is $\ell$-bounded and $\kappa$-spread, and
we replace each $S\in \h$ by $M$ new edges, each consisting of $S$ plus
$\ell-|S|$ new vertices (each used just once), then for large enough M
the resulting $\ell$-graph $\g$ is again $\kappa$-spread.
\end{obs}

\begin{proof}[Derivation of Theorem~\ref{maintheorem} from Theorem~\ref{MT}]
Let $\f$ be as in Theorem~\ref{maintheorem} with $\g$ its set of minimal elements, let $\ell $ with 
$\ell(\f) \leq \ell =O(\ell(\f))$
be large enough that the exceptional probability in Theorem~\ref{MT} is less than 1/4
and let $\nu$ be the $(2q)$-spread probability measure promised by 
Proposition~\ref{Tal6.7}, where $q = q_f(\f)$.  We may assume $\nu$ is supported on $\g$
(since transferring weight from $S$ to $T\sub S$ doesn't destroy the spread condition)
and that $\nu$ takes values in $\QQ$. 
We may then replace $\g$ by $\h$ whose edges are copies of edges of $\g$, and $\nu$ by 
uniform measure on $\h$.

Setting $m= ((2Kq\log \ell)n)$ and $p=2m/n$
(with $n=|X|$ and $K$ as in Theorem~\ref{MT}), we then have (using 
Theorem~\ref{MT} with $\kappa = 1/(2q)$)
\[
\mu_p(\f)\geq  \pr (X_p\in\langle\h\rangle)\geq \pr(|X_p|\geq m)\pr (X_m\in\langle\h\rangle) 
\geq 3\pr(|X_p|\geq m)/4 > 1/2,
\]
implying $p_c(\f)< p=4Kq\log \ell$.
(Note $\h$ $q$-spread with $\0\not\in \h$ implies $q\geq 1/n$, so that
$m$ is somewhat large and $\pr(|X_p|\geq m)\approx 1$.)
\end{proof}
\begin{Remark}\label{Rstupid}
This was done fussily to cover smaller $\ell$ in Theorem~\ref{maintheorem};
if $\ell\ra\infty$, then it gives $\pr (X_p\in\langle\h\rangle)\ra 1$.
\end{Remark}

\section{Main Lemma}\label{SML}

Let $\gc$ be a slightly small 
constant (e.g.\ $\gc =0.1$ suffices), and let $C_0$ be a 
constant large enough to support 
the estimates that follow. Let $\h$ be an $r$-bounded, $\kappa$-spread hypergraph on a set $X$ of size $n$, with $r,\kappa\ge C_0^2$. Set $p=C/\kappa$ with $C_0 \le C \le \kappa/ C_0$ (so $p \leq 1/C_0$), $r'=(1-\gc)r$ and $N=\C{n}{np}$. Finally, fix $\psi:\langle\h\rangle\ra \h$ satisfying $\psi(Z)\sub Z$ for all $Z \in \langle \h \rangle$;
set, for $W\sub X$ and $S\in \h$,
\[
\chi(S,W) = \psi(S\cup W)\sm W;
\]
and say the pair $(S,W)$ is \emph{bad} if $|\chi(S,W)|>r'$ and \emph{good} otherwise.

The heart of our argument is the following lemma (an improvement of~\citep[Lemma~5.7]{ALWZ}),
regarding which a little orientation
may be helpful.
We will (in Theorems~\ref{MT} and \ref{TFS}) be choosing a random subset of $X$ 
in small increments
and would like to say we 
are likely to be making good progress toward containing some $S\in\h$.
Of course such progress is not to be expected
for a \emph{typical} $S$, but this is not the goal:
having chosen a portion $W$ of our eventual set, we just need the 
remainder to contain \emph{some} $S\sm W$, and may focus on those that are more 
likely (basically meaning small).
The key idea (introduced in~\citep{ALWZ} and refined here) is that 
a general $S\sm W$, while not itself small, will,
in consequence of the spread assumption, typically \emph{contain} some small $S'\sm W$.
(In fact $\chi(S,W)$ will usually be one of these: 
an $S'\sm W$ contained in $S\sm W$ will \emph{typically} be small, so we don't need to
steer this choice.)
We then replace each ``good'' $S\sm W$ by $\chi(S,W)$ and iterate, a second nice feature
of the spread condition being that it is not much affected by this substitution.

\begin{lemma}\label{ML}
For $\h$ as above, and $W$ chosen uniformly from $\C{X}{np}$,
\[   
\E [|\{S\in \h:\mbox{$(S,W)$ is bad}\}|] \leq |\h|C^{-r/3}.
\]   
\end{lemma}
\begin{proof}
It is enough to show, for $s\in (r',r]$,
\beq{Expectation1'}
\E \left[|\{S\in \h:\mbox{$(S,W)$ is bad and $|S|=s$}\}| \right] \leq (\gc r)^{-1}|\h|C^{-r/3},
\enq
or, equivalently, that
\beq{Expectation1}
|\{(S,W):\mbox{$(S,W)$ is bad and $|S|=s$}\}| \leq (\gc r)^{-1}N|\h|C^{-r/3}.
\enq
(Note $\gc r=r-r'$ bounds the number of $s$ for which the set in question can be nonempty,
whence the negligible factors $(\gc r)^{-1}$.)

We now use $\h_s =\{S\in \h:|S|=s\}$.
Let $B=\sqrt{C}$ and for $Z\supseteq S\in \h_s$
say $(S,Z)$ is \emph{pathological} if there 
is $T\sub S$ with $t:=|T|>r'$ and
\beq{patho}
|\{S'\in \h_s:S'\in [T,Z]\}| > B^r|\h|\kappa^{-t}p^{s-t}.
\enq
From now on we will always take $Z=W\cup S$ (with $W$ as in Lemma~\ref{ML});
thus $|Z|$ is typically roughly $np$ and, since $\h$ is $\kappa$-spread, 
$|\h|\kappa^{-t} p^{s-t}$ is a natural upper bound on what one might expect for the l.h.s.\ of 
\eqref{patho}.  

Note that in proving \eqref{Expectation1} we may assume $s\leq n/2$:  
we may of course assume $|\h_s|$ is at least the r.h.s.\ of \eqref{Expectation1'}; but then 
for an $S\in \h_s$ of largest multiplicity, say $m$, we have
\[
m\leq \kappa^{-s}|\h|\leq \kappa^{-s}\gc r C^{r/3}|\h_s|
\leq \kappa^{-s}\gc r C^{r/3}m2^n,
\]
which is less than $m$ if $s>n/2$ (since $\kappa>C$).

We bound the nonpathological and pathological parts of~\eqref{Expectation1} separately;
this (with the introduction of ``pathological'') is the source of our improvement over~\citep{ALWZ}.

\mn
\textbf{Nonpathological contributions.}
We first bound the number of $(S,W)$ in~\eqref{Expectation1} with
$(S,Z)$ nonpathological.
This basically follows~\citep{ALWZ}, but ``nonpathological''  allows us to bound the number of
possibilities in Step 3 below 
by the r.h.s.\ of~\eqref{patho}, where~\citep{ALWZ} settles for something like $|\h|\kappa^{-t}$.

\mn
\emph{Step 1.}
There are at most 
\beq{Nbnd} 
\sum_{i =0}^s \binom{n}{np + i} \le \binom {n+s}{np+s} \le N p^{-s}
\enq
choices for
 $Z= W\cup S$.

\mn
\emph{Step 2.}
Given $Z$, let $S'= \psi(Z)$. Choose $T:=S\cap S'$,
for which there are at most $2^{|S'|}\leq 2^r$ possibilities, and set $t=|T|>r'$.
(If $t\leq r'$ then $(S,W)$ cannot be bad, as $\chi(S,W) = S'\sm W\sub T$.)

\mn
\emph{Step 3.}
Since we are only interested in nonpathological choices, the number of possibilities for $S$ is now
at most 
\[B^r|\h|\kappa^{-t}p^{s-t}.\]

\mn
\emph{Step 4.}
Complete the specification of $(S,W)$ by choosing $W\cap S$, 
the number of possibilities for which is at most $2^s$.

\mn

In sum, since $s\leq r$ and $t>r'=(1-\gc)r$,
the number of nonpathological possibilities is at most 
\beq{T1}
2^{r+s}N|\h|B^r (p\kappa)^{-t}
\leq N |\h|(4B)^rC^{-t}   <  N|\h|[4B C^{-(1-\gc)}]^r.
\enq

\mn
\textbf{Pathological contributions.}
We next bound the number of $(S,W)$ as in~\eqref{Expectation1} with $(S,Z)$ pathological. 
The main point here is Step~4.

\mn
\emph{Step 1.}
There are at most $|\h|$ possibilities for $S$.

\mn
\emph{Step 2.}
Choose $T\sub S$ witnessing the pathology of $(S,Z)$ (i.e.\ for which~\eqref{patho} holds);
there are at most $2^s$ possibilities for $T$.

\mn
\emph{Step 3.}
Choose $U\in [T,S]$ for which 
\beq{pathoR}
|\h_s\cap [U,(Z\sm S)\cup U]| > 2^{-(s-t)}B^r|\h|\kappa^{-t}p^{s-t}.
\enq
(Here the left hand side counts members of $\h_s$ in $Z$ whose intersection with $S$ is precisely $U$. 
Of course, existence of $U$ as in~\eqref{pathoR} follows from~\eqref{patho}.) The number of possibilities for this choice is at most $2^{s-t}$.

\mn
\emph{Step 4.}  Choose $Z\sm S$, the number of choices for which is less than $N (2/B)^r$. To see this, write $\Phi$ for the r.h.s.\ of~\eqref{pathoR}.
Noting that $Z\sm S$ must belong to $\C{X\sm S}{np} \cup \C{X\sm S}{np-1} \cup \cdots \cup \C{X\sm S}{np-s}$,
we consider, for $Y$ drawn uniformly from this set, 
\beq{PpathoR}
\pr(|\h_s\cap [U,Y\cup U]|> \Phi) .
\enq
Set $|U|=u$.
We have
\[
|\h_s\cap \langle U\rangle|\leq |\h\cap \langle U\rangle| \leq |\h|\kappa^{-u},
\]
while, for any $S'\in \h_s\cap \langle U\rangle$,
\[
\pr(Y\supseteq S'\sm U)\leq \left(\frac{np}{n-s}\right)^{s-u} 
\]
(of course if $S'\cap S\neq U$ the probability is zero);
so 
\[
\vt:=\E\left[|\h_s\cap [U,Y\cup U]|\right]\leq |\h|\kappa^{-u}\left(\frac{np}{n-s}\right)^{s-u}\leq |\h|\kappa^{-u}\left(2p\right)^{s-u}
\]
(since $n-s \ge n/2$).
Markov's Inequality then bounds the probability in~\eqref{PpathoR} by $\vt/\Phi$, and this bounds the 
number of possibilities for $Z\sm S$ by 
$N(\vt/\Phi)$ (\emph{cf.}~\eqref{Nbnd}),
which is easily seen to be less than $N(2/B)^r$.

\mn
\emph{Step 5.}
Complete the specification of $(S,W)$ by choosing $S\cap W$, which can be done in at most $2^s$ ways.

\mn

Combining (and slightly simplifying), we find that the number of pathological possibilities is at most 
\beq{T2}
|\h|N (16/B)^r.
\enq

Finally, the sum of the bounds in~\eqref{T1} and~\eqref{T2} is less than the 
$(\gc r)^{-1}N|\h|C^{-r/3}$ of \eqref{Expectation1}.\end{proof}


\section{Small uniformities} \label{SJ}

As in~\citep[Lemma 5.9]{ALWZ}, very small set sizes are handled by a simple Janson bound:

\begin{lemma}\label{LJ}
For an $r$-bounded, $\kappa$-spread $\g$ 
on $Y$, 
and $\ga\in (0,1)$,
\beq{JB}
\pr(Y_\ga\not\in \langle\g\rangle) \leq 
\exp\left[- \left(\sum_{t=1}^r\C{r}{t}(\ga \kappa)^{-t}\right)^{-1}\right].
\enq
\end{lemma}
\begin{proof}
We may assume $\g$ is $r$-uniform, 
since modifying it according to Observation~\ref{Ounif}
doesn't decrease the probability in \eqref{JB}.
Denote members of $\g$ by $S_i$ and set 
$\gz_i=\1_{\{Y_\ga\supseteq S_i\}}$. Then
\[
\mu:= \sum \E [\gz_i] =|\g|\ga^r
\]
and
\[
\gL:=
\sum\sum\{\E[\gz_i\gz_j]: S_i\cap S_j\neq \0\} \leq |\g|\sum_{t=1}^r\C{r}{t} \kappa^{-t}|\g|\ga^{2r-t} = \mu^2\sum_{t=1}^r\C{r}{t}(\ga \kappa)^{-t}
\]
(where the inequality holds because $\g$ is $\kappa$-spread),
and Janson's Inequality (e.g.~\citep[Thm.\ 2.18(ii)]{JLR}) bounds 
the probability in~\eqref{JB} by $\exp[-\mu^2/\gL]$.\end{proof}

\begin{corollary}\label{CJ}
Let $\g$ be as in Lemma~\ref{LJ}, 
let $t =\ga |Y|$ be an integer with $\ga\kappa\geq 2r$, and let $W=Y_t$.  Then
\[  
\pr(W\not\in \langle\g\rangle)  \leq 2\exp[-\ga\kappa/(2r)].
\]  
\end{corollary}
\begin{proof}
Lemma~\ref{LJ} gives
\[
\exp[-\ga\kappa/(2r)] \geq \pr(Y_\ga\not\in \langle\g\rangle) 
\geq \pr(|Y_\ga|\leq t)\pr(W\not\in \langle\g\rangle) \geq \pr(W\not\in \langle\g\rangle) /2,
\]
where we use the fact that any binomial $\xi$ with $\E [\xi]\in \Z$ satisfies
$\pr(\xi\leq \E[\xi])\geq 1/2$; see e.g.~\citep{median}.
\end{proof}

\section{Proof of Theorem~\ref{MT}}\label{SMT}
It will be (very slightly) convenient to prove the theorem 
assuming $\h$ is $(2\kappa)$-spread.
Let $\gc$ and $C_0$ be as in Section~\ref{SML} and $\h$ as in the statement of Theorem~\ref{MT}, and recall that asymptotics refer to $\ell$. 
We may of course assume that  $\kappa \ge 2\gamma^{-1}C_0\log \ell$ (or the result is trivial with
a suitably adjusted $K$).

Fix an ordering ``$\prec$'' of $\h$.  
In what follows we will have a sequence $\h_i$, with $\h_0=\h$ and 
\[
\h_i\sub \{\chi_i(S,W_i): S\in \h_{i-1}\},
\]
where $W_i$ and $\chi_i$ will be defined below (with $\chi_i$ a version of the $\chi$
of Section~\ref{SML}).  We then order $\h_i$ by setting
\[
\chi_i(S,W_i)\prec_i\chi_i(S',W_i) \Lra S\prec_{i-1} S'.
\]
(So each member of $\h_i$ ultimately inherits its position in $\prec_i$ from some member of $\h$.
This is not very important: we will be applying Lemma~\ref{ML} repeatedly, and the
present convention just provides a concrete $\psi$ for each stage of the iteration.)

Set $C=C_0$ and $p= C/\kappa$, define $m$ by $(1-\gc)^m = \sqrt{\log \ell}/\ell$, and set $q= \log \ell/\kappa$. Then $\gc^{-1}\log \ell \sim m \le \gc^{-1}\log \ell$ and Theorem~\ref{MT} 
will follow from the next assertion.
\begin{claim}\label{toshow}
If W is a uniform $((mp+q)n)$-subset of $X$, then $W \in \langle\h\rangle$ w.h.p.
\end{claim}
\begin{proof}

Set $\gd = 1/(2m)$.  Let $r_0=\ell$ and $r_i= (1-\gc)r_{i-1} = (1-\gc)^ir_0$ for $i\in [m]$.
Let $X_0=X$ and, for $i=1\dots m$, let $W_i$ be uniform from $\C{X_{i-1}}{np}$
and set $X_i=X_{i-1}\sm W_i$. 
(Note the assumption $\kappa \ge 2\gamma^{-1}C_0\log \ell$ ensures $|X_m| \ge n/2$.)

For $S\in \h_{i-1}$ let $\chi_i(S,W_i)=S'\sm W_i$, where $S'$ is the first member of $\h_{i-1}$
contained in $W_i\cup S$ (with $\h_{i-1}$ ordered by $\prec_{i-1}$).
Say $S$ is \emph{good} if $|\chi_i(S,W_i)|\leq r_i$ (and \emph{bad} otherwise),
and set 
\[
\h_i=\{\chi_i(S,W_i):\mbox{$S\in \h_{i-1}$ is good}\}.
\]
Thus $\h_i$ is an $r_i$-bounded collection of subsets of $X_i$ and inherits the ordering $\prec_i$
as described above.

Finally, choose $W_{m+1}$ uniformly from $\C{X_m}{nq}$.  Then $W:=W_1\cup\cdots\cup W_{m+1}$ is as in Claim~\ref{toshow}.
Note also that $W\in\langle\h\rangle$ whenever $W_{m+1}\in \langle\h_m\rangle$.
(More generally, $W_1\cup\cdots\cup W_i\cup Y\in \langle\h\rangle$ whenever $Y\sub X_i$ 
lies in $\langle\h_i\rangle$.)

So to prove the claim, we just need to show
\beq{toshow'}
\pr(W_{m+1}\in \langle\h_m\rangle) = 1-o(1)
\enq
(where the $\pr$ refers to the entire sequence $W_1\dots W_{m+1}$).

For $i\in [m]$ call $W_i$ \emph{successful} if $|\h_i|\geq (1-\gd)|\h_{i-1}|$,
call $W_{m+1}$ successful if it lies in $\langle\h_m\rangle$, and
say a sequence of $W_i$'s is successful if each of its entries is.
We show  a little more than~\eqref{toshow'}:
\beq{toshow''}
\pr(\mbox{$W_1\dots W_{m+1}$ is successful}) = 1-\exp\left[-\gO(\sqrt{\log \ell})\right].
\enq

For $i\in [m]$, according to Lemma~\ref{ML} (and Markov's Inequality), 
\[
\pr(\mbox{$W_i$ is \emph{not} successful}\,|\,\mbox{$W_1\dots W_{i-1}$ is successful})  
< \gd^{-1}C^{-r_{i-1}/3},
\]
since $W_1\dots W_{i-1} $ successful implies
$
|\h_{i-1}| > (1-\gd)^m|\h|>|\h|/2, 
$
which, 
since $|\h_{i-1}\cap \langle I\rangle|\leq |\h\cap \langle I\rangle|$ and we assume $\h$
is $(2\kappa)$-spread),
gives the spread condition~\eqref{spread} for $\h_{i-1}$.
Thus
\beq{Pbd1}
\pr(\mbox{$W_1\dots W_m$ is successful}) > 1-\gd^{-1}\sum_{i=1}^mC^{-r_{i-1}/3} ~
 > 1-\exp\left[-\sqrt{\log \ell}\right]
\enq
(using $r_m=\sqrt{\log\ell}$).

Finally, if $W_1\dots W_m$ is successful, then Corollary~\ref{CJ} 
(applied with $\g=\h_m$, $Y=X_m$, $\ga=nq/|Y|\geq q$,
$r=r_m $, 
and $W=W_{m+1}$)
gives 
\beq{Pbd2}
\pr(W_{m+1}\not\in \langle\h_m\rangle)\leq 2\exp\left[-\sqrt{\log \ell} / 2\right] ,
\enq
and we have \eqref{toshow''} and the claim.
\end{proof}

\section{Proof of Theorem~\ref{TFS}}\label{SFS}

We assume the setup of Theorem~\ref{TFS} with
$\gc$ and $C_0$ as in Section~\ref{SML} and $\kappa \ge C_0^2$ (or there is nothing to prove).
We may assume $\h$ is $\ell$-uniform, since the construction 
of Observation~\ref{Ounif} 
produces an $\ell$-uniform, $\kappa$-spread $\g$ with $\xi_\g\geq \xi_\h$. In particular this gives
\beq{lmax}
|\h|\ell =\sum_{x\in X}|\h\cap \langle x\rangle|\leq n\kappa^{-1}|\h|.
\enq

We first assume $\kappa$ is \emph{slightly} large, precisely
\beq{bigK}
\kappa \geq \log ^3\ell;
\enq
the similar but easier argument for smaller values will be given at the end.
(The bound in \eqref{bigK} is convenient but there is nothing delicate about this choice.)

\begin{claim}\label{CFS}
For $\kappa$ as in \eqref{bigK} and $C_0\leq C \le \gc\kappa/(4\log\ell)$,
\[
\pr(\xi_\h> (3C/\gc)\ell/\kappa) < \exp[-(\log \ell \log C)/4].
\]
\end{claim}

\begin{proof}[Proof of Theorem~\ref{TFS}  in regime \eqref{bigK} given Claim~\ref{CFS}]
The ``w.h.p.'' statement is immediate (take $C=C_0$).  
For the expectation, $Z_\h$, 
set $t=(3C_0/\gc)\ell/\kappa$ and $T=3\ell/(4\log \ell)$.
By Claim~\ref{CFS} we have, for all $x\in [t,T]$,
\[
\pr(\xi_\h> x) \leq f(x):= \exp\left[-\log \ell \log(\gc \kappa x/3\ell)/4\right] = (bx)^a = b^a x^a,
\]
where $a = -(\log\ell)/4$ and $b = \gamma \kappa/3\ell$.
Noting that $\xi_\h\leq\ell$, we then have
\[
Z_\h\leq t + \int_t^T\pr(\xi_\h>x)dx +\ell \pr(\xi_\h>T)
\leq   t +\int_t^T f(x)dx + \ell f(T) =O(\ell/\kappa).
\]
Here $t=O(\ell/\kappa)$ and the other terms are much smaller:
the integral is less than $-1/(a+1)b^at^{a+1} =O(1/\log \ell)C_0^at~$, while \eqref{bigK} 
easily implies that $f(T) =(\gc\kappa/(4\log\ell))^a$ is $o(1/\kappa)$.
\end{proof}

\begin{proof}[Proof of Claim~\ref{CFS}]

Terms not defined here (beginning with $p=C/\kappa$ and $W_i$;
note $C$ is now as in Claim~\ref{CFS}, rather than set to $C_0$) are as in Section~\ref{SMT},
but we (re)define $m$ by 
 $(1-\gc)^m = \log \ell/\ell$ and set $q=\log C\log^2 \ell/\kappa$, 
noting that~\eqref{lmax} gives $p\geq C\ell /n$.

\nin

It's now convenient to generate the $W_i$'s 
using the $\xi_x$'s in the natural way:  let
\[
a_i=\left\{\begin{array}{ll}
(ip)n & \mbox{if $i\in \{0\}\cup [m]$,}\\
(mp+q)n& \mbox{if $i=m+1$,}
\end{array}\right.
\]
and let $W_i$ consist of the $x$'s in positions $a_{i-1}+1\dots a_i$ 
when $X$ is ordered according to the $\xi_x$'s.

\begin{proposition}\label{Wobs}
With probability $1-e^{-\gO(C\ell)}$,
\beq{Wobs'}
\xi_x\leq \eps_i:=\left\{\begin{array}{ll}
2ip & \mbox{if $i\in \{0\}\cup [m]$}\\
2(mp+q)& \mbox{if $i=m+1$}
\end{array}\right\}
~\mbox{for all $i$ and $x\in W_i$.}
\enq
\end{proposition}
\begin{proof}
Failure at $i\geq 1$ implies 
\beq{ai}
|\xi^{-1}[0,\eps_i]|< a_i.
\enq
But $|\xi^{-1}[0,\eps_i]|$ is binomial with mean $\eps_i n =2 a_i\geq 2C\ell$,
so the probability that~\eqref{ai} occurs for some $i$ is less than $\exp[-\gO(C\ell)] $
(see e.g.~\citep[Theorem 2.1]{JLR}).
\end{proof}

We now write $\ov{W}_i$ for $W_1\cup\cdots\cup W_i$.
\begin{proposition}\label{PS}
If $W_{m+1}\in \langle\h_m\rangle$, then $W$ contains some $S\in \h$ with 
\[   
|S\sm \ov{W}_i|\leq r_i~\forall i\in [m].
\]   
\end{proposition}
\begin{proof}
Suppose $W\supseteq S_m\in \h_m$.  By construction (of the $\h_i$'s) there are $S_{m-1}\dots S_1,S_0=:S$ 
with $S_i\in \h_i$ and $S_i=S_{i-1}\sm W_i$, whence
$S_i=S\sm \ov{W}_i$ for $i\in [m]$; and $S_i\in \h_i$ then gives the proposition.
\end{proof}

We now define ``success'' for $(\xi_x:x\in X)$
to mean that $W_1\dots W_{m+1}$ is successful in our earlier sense
\emph{and} \eqref{Wobs'} holds.
Notice that with our current values of $m$ and $q$ (and $r_m=\ell (1-\gc)^m = \log \ell$), we can replace 
the error terms in~\eqref{Pbd1} and~\eqref{Pbd2} by essentially
$\gd^{-1}C^{-\log\ell/3}$ and $e^{-\log C\log \ell/2}$, which with Proposition~\ref{Wobs} 
bounds the probability 
that $(\xi_x:x\in X)$ is \emph{not} successful by (say) $\exp[-(\log \ell \log C)/4]$. 

We finish with the following observation.

\begin{proposition}\label{FinalP}
If $(\xi_x:x\in X)$ is successful then $\xi_\h \leq (3C/\gc)\ell/\kappa$.
\end{proposition}
\begin{proof}
For $S$ as in Proposition~\ref{PS}, we have (with $W_0=\0$ and $\eps_0=0$)
\begin{align*}
\xi_\h &\leq \sum_{i=1}^{m+1} \eps_i|S\cap W_i|
~=~ 
\sum_{i=1}^{m+1} (\eps_i-\eps_{i-1})|S\sm \ov{W}_{i-1}|\\
&\leq
2\left[\sum_{i= 1}^m(1-\gc)^{i-1} p+ (1-\gc)^mq\right] \ell \\
&\leq 2[C/(\gc\kappa) + (\log \ell/\ell)(\log C\log^2\ell/\kappa)]\ell< (3C/\gc)\ell/\kappa. \qedhere
\end{align*}\end{proof}
This completes the proof of Claim~\ref{CFS} (and of
Theorem~\ref{TFS} when $\kappa $ satisfies \eqref{bigK}).
\end{proof}

Finally, for $\kappa$ below the bound in \eqref{bigK}
(actually, for $\kappa$ up to about $\ell/\log \ell$), a subset of the preceding argument suffices.
We proceed as before, but now only with $C=C_0$ (so $p=C_0/\kappa$), stopping at $m$
defined by $(1-\gc)^m = 1/\kappa$ (so $m\approx \gc^{-1}\log \kappa$).
The main difference here is that there is no ``Janson'' phase:
$W_1\dots W_m$ is successful with probability $1-\exp[-\gO(\ell/\kappa)]$, 
and when it \emph{is} successful we have (as in the proof of Proposition~\ref{FinalP},
now just taking $W_{m+1}=X\sm \ov{W}_m$)
\[
\xi_\h\leq 
\sum_{i=1}^m (\eps_i-\eps_{i-1})|S\sm \ov{W}_{i-1}| +|S\cap W_{m+1}| 
< 2(C_0/(\gc\kappa))\ell +\ell/\kappa
\]
(so also $Z_\h\leq O(\ell/\kappa) + \exp[-\gO(\ell/\kappa)]\ell = O(\ell/\kappa)$).

\section{Applications}\label{Applications}

Much of the significance of Theorem~\ref{maintheorem}---and of the skepticism with which 
Conjecture~\ref{CKK} was
viewed in~\citep{KK}---derives from the strength of their 
consequences, a few of which we discuss 
(\emph{briefly}) here.

For this discussion,
$\K^r_n=\C{V}{r}$ is
the complete $r$-graph on $V=[n]$,  
and 
$\h^r_{n,p}$ is the $r$-uniform counterpart of the usual binomial random graph $G_{n,p}$.
Given $r,n$ and an $r$-graph $H$, we use $\g_H$ for the collection of 
(unlabeled) copies of $H$ in $\K^r_n$ and $\f_H$
for $\langle\g_H\rangle$.  As usual, $\gD$ is maximum degree.

As noted earlier, Conjecture~\ref{CKK} was motivated especially by Shamir's Problem (since 
resolved in~\citep{JKV}), and the conjecture that became Montgomery's theorem~\citep{Montgomery}.
Very briefly:
for $n$ running over multiples of a given (fixed) $r$, Shamir's Problem asks for 
estimation of $p_c(\f_H)$ when $H$ is a perfect matching ($n/r$ disjoint edges),
and~\citep{JKV} proves the natural conjecture that this threshold is $\Theta(n^{-(r-1)}\log n)$;
and~\citep{Montgomery} shows that for fixed $d$, the threshold for $G_{n,p}$ to contain 
a given $n$-vertex tree 
with maximum degree $d$ is $\Theta(n^{-1}\log n)$,
where the implied constant in the upper bound depends on $d$ (though it probably shouldn't).
See~\citep{JKV,Montgomery} for some account of the history of these problems.
In both cases---and in most of the other examples mentioned following Theorem~\ref{TMr} (all but the one from~\citep{Kriv})---the 
lower bounds derive from the coupon-collectorish requirement that the (hyper)edges cover 
the vertices, and it is the upper bounds that are of interest.

In fact, Theorem~\ref{maintheorem} gives not just Montgomery's theorem, but its natural extension 
to $r$-graphs and more.
(Strictly speaking, Montgomery proves more than the original conjecture---see 
Section~\ref{Concluding}---and we are not so far recovering this stronger result.)
Say an $r$-graph $F$ is a \emph{forest} if it contains no \emph{cycle}, meaning distinct vertices 
$v_1\dots v_k$ and distinct edges $e_1\dots e_k$ such that $v_{i-1}, v_i\in e_i$ $\forall i$ 
(with subscripts mod $k$).
A \emph{spanning tree} is then a forest of size $(n-1)/(r-1)$.
For a (general) $r$-graph $F$, let $\rho(F) $ be the maximum size of a forest in $F$ and set
\[
\vp(F) = \max\{  1-\rho(F')/|F'|  :  \0\neq F'\sub F\}.
\]
\begin{theorem}\label{TMr}
For each $r$ and $c$ there is a $K$ such that if $H$ is an $r$-graph on $[n]$ with $\gD(H)\leq d$ and 
$\vp(H) \leq c/\log n$, then 
\[
p_c(\f_H) <Kdn^{-(r-1)}\log |H|. 
\]
\end{theorem}

\nin
This gives $p_c(\f_H)=\Theta(n^{-(r-1)}\log n)$ if $H$ is a perfect matching
(as in Shamir's Problem), 
or
a ``loose Hamiltonian cycle'' (a result of~\citep{DFLS}, to which we refer for definitions and
history of the problem),
and $p_c(\f_H) <Kdn^{-(r-1)}\log n$ if $H$ is a
spanning tree with $\gD(H)\leq d$.  For fixed $d$ the latter
is the aforementioned $r$-graph generalization of~\citep{Montgomery}
(or a slight improvement thereof in that the dependence on $d$---which, again,
is probably unnecessary---is explicit),
and for $d=n^{\gO(1)}$ it is a result of Krivelevich~\citep[Theorem 1]{Kriv},
which is again tight up to the value of $K$ (see~\citep[Theorem 2]{Kriv}).

\mn

The last application we discuss here was suggested to us by Simon Griffiths and  Rob Morris. Set $c_d =(d!)^{2/(d(d+1))}$ and
$p^*(d,n) =c_dn^{-2/(d+1)}(\log n)^{2/(d(d+1))}.$

\begin{theorem}\label{TDelta}
For fixed d and H any graph on $[n]$ with $\gD(H)\leq d$, 
\beq{d.asymp}
p_c(\f_H) < (1+o(1))p^*(d,n).
\enq
\end{theorem}
\nin

When $(d+1)\,|\,n$ and 
$H$ is a $K_{d+1}$-factor (that is, $n/(d+1)$ disjoint $K_{d+1}$'s),
$p^*(d,n)$ is the asymptotic value of $p_c(\f_H)$. Here \eqref{d.asymp} with $O(1)$ in place of $1+o(1)$ was proved in
\citep{JKV}, while the asymptotics are given by the combination of~\citep{HT2} and~\citep{Riordan,Heckel}; we state this in a form convenient for use below:
\begin{theorem}\label{TKRH}
For fixed d and $\eps>0$, and n ranging over multiples of $d+1$, 
if $p> (1+\eps)p^*(d,n)$, then
$G_{n,p}$ contains a $K_{d+1}$-factor w.h.p. \qed
\end{theorem}

Interest in $p_c(\f_H)$ for $H$ as in Theorem~\ref{TDelta} dates to at least 1992, when Alon and F\"uredi~\citep{AlonF} showed the upper bound $O(n^{-1/d}(\log n)^{1/d})$, and has intensified since~\citep{JKV}, motivated by the idea that $K_{d+1}$-factors should be the worst case. See~\citep{bdd2,bdd1} for history and the most recent results; with $O(1)$ in place of $1+o(1)$, Theorem~\ref{TDelta} is conjectured in~\citep{bdd1} and in the stronger ``universal'' form 
(see Section~\ref{Concluding}) in~\citep{bdd2}.

{Theorem~\ref{TKRH} probably extends to $r$-graphs and $d$ of 
the form $\C{s-1}{r-1}$.  This just needs extension of Theorem~1 of~\citep{Riordan} to 
$r$-graphs (suggested at the end of~\citep{Riordan}), which (with~\citep{HT2}) 
would give asymptotics of the threshold for $\h^r_{n,p}$ to contain a $\K^r_s$-factor
(where $\K^r_s$, recall, is the complete $r$ graph on $s$ vertices).}

\mn

Each of Theorems~\ref{TMr} and \ref{TDelta} begins with the following easy observations.
(The first, an approximate converse of Proposition~\ref{Tal6.7}, is
the trivial direction of LP duality.)
\begin{obs}
If an increasing $\f$ supports a $q$-spread measure, then $q_f(\f)< q$.
\end{obs}
\nin
(More precisely, $q_f(\f)$ is the least $q$ such that $\f$ supports a probability measure $\nu$
with $\nu(\langle S\rangle)\leq 2q^{|S|}$ $\forall S$.)
\begin{obs}\label{Sobs}
Uniform measure on $\g_H$ is $q$-spread if and only if:
for $S\sub \K^r_n$ isomorphic to a subhypergraph of $H$, 
$\gs$ a uniformly random permutation of $V$,
and $H_0\sub \K_n^r$ a given copy of $H$,
\beq{Psig}
\pr(\gs(S)\sub H_0) \leq q^{|S|}.
\enq
\end{obs}

Proving
Theorem~\ref{TMr} is now just a matter of verifying \eqref{Psig}
with $q=O(dn^{-(r-1)})$, which we leave to the reader.
(It is similar to the proof of \eqref{dsp1}.)

\begin{proof}[Proof of Theorem~\ref{TDelta}]
The next assertion is the main thing we need to check here.   
\begin{lemma}\label{LDelta}
There is $\eps=\eps_d>0$ such that
if $H$ is as in Theorem~\ref{TDelta} and has no component isomorphic to $K_{d+1}$, then 
\beq{qfF}
q_f(\f_H) \leq n^{- (2/(d+1)+\eps)}=:q.
\enq
\end{lemma}
\begin{proof}
We just need to show \eqref{Psig} for $q$ as in \eqref{qfF} and 
$S,H_0$ as in Observation~\ref{Sobs}, 
say with $W=V(S)$, $s=|S|$, and $f$ the size of a spanning forest of $S$.  
We may of course assume $S$ has no isolated vertices, so $w:=|W|\leq 2f$.  
We show 
\beq{dsp1}
\pr(\gs(S)\sub H_0) < (e^2d/n)^f
\enq
and
\beq{dsp2}
\frac{f}{s}\geq \frac{2(d+1)}{(d+2)d} = \frac{2}{d+1}+\eps_0,
\enq
where $\eps_0 = 1/((d+2)(d+1)d)$, implying that for any $\eps< \eps_0$, \eqref{Psig}
holds for large enough $n$.
\begin{proof}[Proof of \eqref{dsp1}] 
Let $\ga, \gb:W\ra V$ be, respectively, a uniform injection and a uniform map.  Then\begin{align*}
(d/n)^f&\geq  \pr(\gb(S)\sub H_0) ~ \geq ~ \pr(\mbox{$\gb$ is injective})
\pr(\gb(S)\sub H_0|\mbox{$\gb$ is injective})\\
&= (n)_wn^{-w}\pr(\ga(S)\sub H_0) ~> ~ e^{-2f}\pr(\gs(S)\sub H_0).\qedhere
\end{align*}
\end{proof}

\begin{proof}[Proof of \eqref{dsp2}] 
We may of course assume $S$ is connected, in which case we have $f=w-1$ and
upper bounds on $s$:  $\C{w}{2}$ if $w\leq d$; 
$\C{d+1}{2}-1$ if $w=d+1$; and $wd/2$ if $w\geq d+2$.
The corresponding lower bounds on $f/s$ are $2/d$, $2d/((d+2)(d+1)-2)$ and 
$2(d+1)/((d+2)d)$, the smallest of which is the last.
\end{proof}
This completes the proof of Lemma~\ref{LDelta}. 
\end{proof}
  
We are now ready for Theorem~\ref{TDelta}.  
Let $\vs=\vs_n$ be some slow $o(1)$ (e.g.\ $1/\log  n$).   
By Theorem~\ref{TKRH} there is $p_1\sim p^*(d,n)$
such that if $(d+1)\,|\,m>(1-\vs)n$ then $G_{m,p_1}$ contains a $K_{d+1}$-factor w.h.p.,
while by Lemma~\ref{LDelta} and Theorem~\ref{maintheorem} (or, more precisely,  
Remark~\ref{Rstupid}), there is $p_2$ with $p^*(d,n)\gg p_2\gg n^{-(2/(d+1)+\eps)}$ 
such that if $m\geq \vs n$ then for any given $m$-vertex $H'$ with $\gD(H')\leq d$,
$G_{m,p_2}$ contains (a copy of) $ H'$ w.h.p.

Let $H_1$ be the union of the copies of $K_{d+1}$ in $H$ (each of which must be a component of $H$),
$H_2=H-H_1$, and $n_i=|V(H_i)|$ (so $n_1+n_2=n$).  Let $G_1\sim G_{n,p_1}$ and
$G_2\sim G_{n,p_2}$ be independent on the common vertex set $V=[n]$ and
$G=G_1\cup G_2$.  Then $G\sim G_{n,p}$ with $p=1-(1-p_1)(1-p_2)\sim p^*(d,n)$,
and we just need to show $G\supseteq H$ w.h.p.  In fact we find each $H_i$ in the 
corresponding $G_i$, in order depending on $n_2$:
if $n_2\geq \vs n$, then w.h.p.\ $G_1$ contains $H_1$, say on vertex set $V_1$, and 
w.h.p.\ $G_2[V\sm V_1]$ contains $H_2$; and
if $n_2< \vs n$, then w.h.p.\ $G_2$ contains $H_2$ on some $V_2$, and 
w.h.p.\ $G_1[V\sm V_2]$ contains $H_1$.
\end{proof}

\section{Concluding remarks}\label{Concluding}

In closing we briefly mention (or recall) a few unresolved issues related to the present work.

\nin
\textbf{A.}
First, of course, it would be nice to prove Conjecture~\ref{littleTal}, which is now equivalent to Conjecture~\ref{CKK}.  

\nin
\textbf{B.}
It would be interesting to understand whether, in Shamir's and related problems, 
the $\log \ell$ emerging from our argument somehow reflects the coupon-collector
requirement (edges cover vertices) that drives the lower bounds.
Partly as a way of testing this, one might try to see if 
the present machinery can be extended to apply directly 
(rather than \emph{via}~\citep{Riordan,Heckel}) to
questions where coupon-collector considerations (correctly) predict a smaller gap, 
as in the fractional powers of $\log n$ in Theorem~\ref{TKRH}.

\nin
\textbf{C.}
The arguments of~\citep{Montgomery} and~\citep{bdd1} give stronger
``universality'' results; e.g.~\citep{Montgomery} says that the 
appropriate $G_{n,p}$ w.h.p.\ contains \emph{every} tree respecting the degree bound.
Whether this can be proved along present lines remains unclear;
if so, it would seem to be more a question of managing \emph{some} understanding of the class 
of universal graphs (with, of course, a view to the spread)
than of extending Theorem~\ref{maintheorem}.

\nin
\textbf{D.}
As mentioned following Corollary~\ref{FSbd}, what prevents us from extending 
to other values of the dimension $k$ is inadequate control of the spread.
(Here it doesn't really matter whether we think of ``assignments'' or of the threshold for 
containing a member of the $\h$ in \eqref{xiH}.)
The difficulty is the same for the related problem of thresholds for existence of
designs.  We don't have anything to suggest in the way of a remedy and just indicate one issue,
for simplicity sticking to Steiner triple systems (STS's; see~\citep{vw} for background);
thus $X=\K^3_n$ 
(with $n\equiv 1$ or $3$ $\pmod{6}$), $\h$ is the hypergraph of STS's, 
and for the spread (which should be $\Theta(1/n)$),
we may take 
\beq{KSTS}
\kappa = \min_{S\sub X}\left(|\h|/|\h\cap\langle S\rangle|\right)^{1/|S|}.
\enq
Results of Linial and Luria~\citep{LiLu} (upper bound) and Keevash~\citep{Keevash2} (lower bound) 
give 
\beq{LLK}
|\h| =((1+o(1))n/e^2)^{n^2/6}.
\enq
Viewed enumeratively this is very satisfactory, having been an old conjecture of Wilson~\citep{Wilson}.  But for present purposes, even ignoring our weaker understanding 
of $|\h\cap\langle S\rangle|$ 
(the number of completions of a partial STS $S$), it 
is not enough:  even if this quantity is, as one expects, roughly $(n/e^2)^{n^2/6-|S|}$,
the r.h.s.\ of \eqref{KSTS} can be dominated by the ``error'' factor $(1+o(1))^{n^2/(6|S|)}$ if
$S$ is slightly small and the $o(1)$ in \eqref{LLK} is negative.

\nin
\textbf{E.}
Finally, we recall a related conjecture from~\citep{KK}
(stated there only for graphs, but this shouldn't matter).
For $\f=\f_H$ as in Section~\ref{Applications}, let
$p_{\Ee}(\f)$ be the least $p$ such that for every $H'\sub H$
the expected number of (unlabeled) copies of $H'$ in $\h^r_{n,p}$
is at least 1.  Then $p_{\Ee}(\f)/2$ is again a trivial lower bound on $p_c(\f)$---and,
where it makes sense, probably more intuitive than $q(\f)$ or $q_f(\f)$---and 
from~\citep[Conjecture 2]{KK} we have:
\begin{conjecture}\label{CKK2}
There is a universal $K$ such that for every $\f=\f_H$ as above,
\[   
p_c(\f) \le Kp_{\Ee}(\f)\log |X|.
\]   
\end{conjecture}
\nin
Again, we can presumably replace $\log |X|$ by $\log |H|$, as would now follow from 
a positive answer to the obvious question:  do we always have $q_f(\f) =O(p_{\Ee}(\f))$?

\section*{Acknowledgments}
The first, second and fourth authors were supported by NSF grant DMS-1501962 and BSF Grant 2014290.
The third author was supported by NSF grant DMS-1800521.

\bibliographystyle{amsplain}
\bibliography{frac_thresh}

\end{document}